\documentclass[12pt]{article}

\usepackage{overpic,graphicx}
\usepackage{amssymb,amsfonts,amsmath,amsthm}
\usepackage{animate}
\usepackage{color}
\usepackage[english]{babel}
\usepackage[latin1]{inputenc}
\usepackage{times}
\usepackage{geometry}
\usepackage[breaklinks,pdfstartview=FitH,colorlinks,linktocpage,bookmarks=true]{hyperref}

\newtheorem{thm}{Theorem}

\newtheorem{lem}[thm]{Lemma}

\newtheorem{mainthm}[thm]{Main Theorem}
\newtheorem{splitting-lemma}[thm]{Splitting Lemma}

\def\deg{\mathop{\mathrm{deg}}\nolimits}

\long\def\comment#1\endcomment{}

\begin{document}


\title{Surfaces containing two circles through each point and 
decomposition of quaternionic matrices}

\author{A. Pakharev, M. Skopenkov}
\date{}

\maketitle

\begin{abstract}
We find all analytic surfaces in space $\mathbb{R}^3$
such that through each point of the surface one can draw two circular arcs fully contained in the surface.
The proof uses a new decomposition technique for quaternionic matrices.

\smallskip

\noindent{\bf Keywords}: circle, Moebius geometry, quaternion, Pythagorean n-tuple, matrix decomposition

\noindent{\bf 2010 MSC}:
\end{abstract}

\footnotetext[0]{
The research was carried out at the IITP RAS at the expense of the Russian
Foundation for Sciences (project N14-50-00150). The authors are grateful to R. Krasauskas and N. Lubbes for the movie in Figure~\ref{movie}.
}

We find all surfaces in space $\mathbb{R}^3$ 
such that through each point of the surface one can draw two circular arcs fully contained in the surface. 
Due to natural statement and obvious architectural motivation, this is a problem which \emph{must} be solved by mathematicians. 
It turned out to be hard and remained open in spite of many partial advances starting from the works of Darboux from the XIX century, e.g., see \cite{Niels-13,coolidge:1916,NS11}.

\begin{mainthm}\label{mainthm}
If through each point of an analytic surface in $\mathbb{R}^3$ one can draw two transversal circular arcs fully contained in the surface (and analytically depending on the point) then the surface is an image of a subset of one of the following sets under some composition of inversions:
\begin{itemize}
\item[(E)]\label{ETS} the set $\{\,p+q:p\in\alpha,q\in\beta\,\}$,  where $\alpha,\beta$ are two circles in $\mathbb{R}^3$;
\item[(C)]\label{CTS} the stereographic projection of the 
set $\{\,p\cdot q:p\in\alpha,q\in\beta\,\}$,  where $\alpha,\beta$ are two circles in the sphere $S^3$ identified with the set of unit quaternions; 
\item[(D)]\label{DC} the stereographic projection of the intersection of $S^3$ with another 3-dimensional quadric.
\end{itemize}
\end{mainthm}

\vspace{-0.4cm}

\begin{figure}[hbt]
\begin{center}
\begin{tabular}{c}
\includegraphics[height=2.2cm]{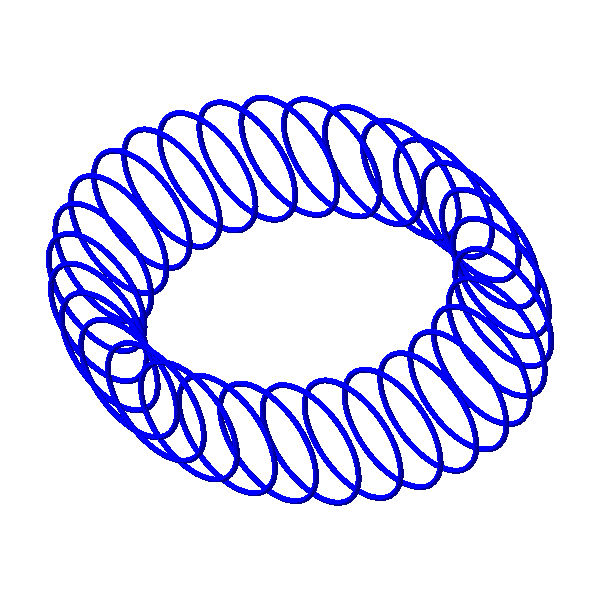}
\end{tabular}
\hspace{-0.5cm}
\begin{tabular}{c}
\includegraphics[height=2.2cm]{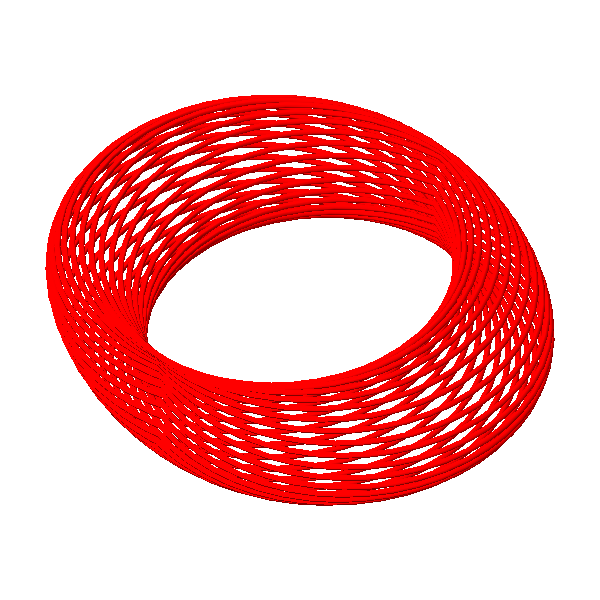}
\end{tabular}
\begin{tabular}{c}
\includegraphics[height=1.8cm]{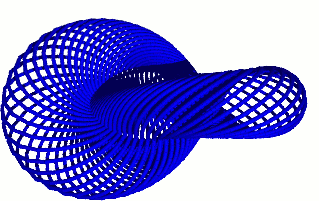}
\end{tabular}
\hspace{-0.5cm}
\begin{tabular}{c}
\includegraphics[height=1.8cm]{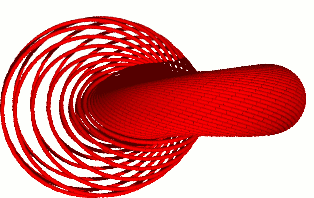}
\end{tabular}
\begin{tabular}{c}
\includegraphics[width=0.15\textwidth]{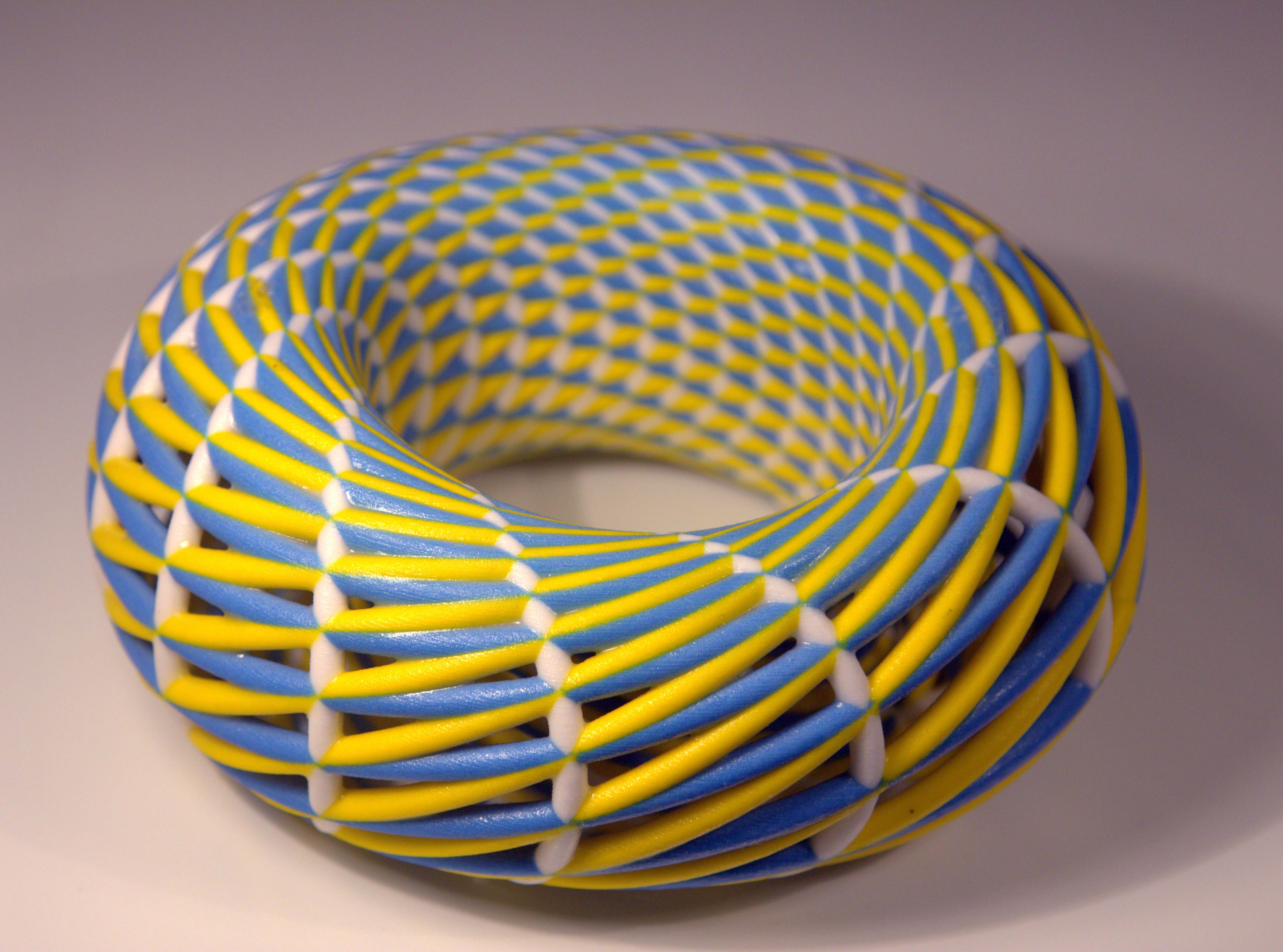}
\end{tabular}	
\end{center}
\vspace{-0.8cm}
\caption{
Euclidean (E) and 
Clifford  (C) translational surfaces, and a Darboux cyclide (D).}
\label{movie}
\end{figure}

Let us give a general plan of the proof of the theorem and prove one particular lemma. 
We solve a more general problem of finding all such surfaces in $S^4$ instead of $\mathbb{R}^3$. 
Using Schicho's parametrization of surfaces containing two conic sections through each point \cite{schicho:2001} we reduce the problem 
to solving the equation
$
X_1^2+X_2^2+X_3^2+X_4^2+X_5^2=X_6^2
$ 
in polynomials $X_{1},\dots,X_6\in\mathbb{R}[u,v]$ of degree at most $2$ in each of the variables $u$ and $v$. 
Such ``Pythagorean $6$-tuple'' of polynomials defines a surface $X_1(u,v):\dots:X_6(u,v)$ in $S^4$ containing two (possibly degenerate) circles $u=\mathrm{const}$ and $v=\mathrm{const}$ through each point. 

Our approach uses decomposition of quaternionic matrices. 
The above equation holds if and only if the matrix 
{ $\left(\begin{smallmatrix}X_6-X_5 & X_1+iX_2+jX_3+kX_4\\ X_1-iX_2-jX_3-kX_4 & X_6+X_5\end{smallmatrix}\right)$}
is \emph{degenerate}, i.e., has left-linearly dependent rows. 
A $2\times 2$ matrix $M$ \emph{splits}, 
if it is a Kronecker product of two vectors, 
i.e., $M_{ij}=x_iy_j$ for $1\le i,j,\le 2$ and some $x_1,x_2,y_2,y_2\in\mathbb{H}[u,v]$. 
Each splitting matrix must be degenerate, and over a commutative unique factorization domain the converse is also true. 
Several results measuring the relation between the two notions over $\mathbb{H}[u,v]$ lead to the solution of our equation.
We prove just one now. 

\begin{lem} \label{l-adding variable} Each degenerate matrix with the entries from $\mathbb{H}[u,v]$ of degree at most $1$ in $v$ splits.
\end{lem}


\begin{proof}
Consider the following $2$ cases.

Case 1: 
The matrix entries are not dependent of $v$.
In this case we prove the lemma by induction over the minimal degree of the entries using division with remainders in $\mathbb{H}[u]$ \cite{Ore}. 

The base of induction is the case when one of the entries vanishes, say, $M_{22}=0$. 
Since the matrix is degenerate it follows that then $M_{12}M_{21}=0$. Without loss of generality we have $M_{21}=0$. 
Thus the matrix splits as $M=(1,0)\otimes (M_{11},M_{12})$.

To perform induction step, assume that each entry is nonzero.
Without loss of generality let $M_{22}$ be the entry with the minimal degree. 
Divide $M_{21}$ by $M_{22}$ from the left with remainders: $M_{21}=M_{22}X+M'_{21}$, where $X,M'_{21}\in\mathbb{H}[u]$ and $\deg M'_{21}<\deg M_{22}$. 
Subtract the second column of the matrix $M$ right-multiplied by $X$ from the first column. 
The resulting matrix is also degenerate and has the entry $M'_{21}$ of degree smaller than the degree of $M_{22}$. By the inductive hypothesis, the resulting matrix splits. Thus the initial matrix $M$ splits.

Case 2: Some entries of the matrix depend nontrivially of $v$. 
In this case we prove the lemma by induction over the number of such entries.

The base of induction is the case when there is only one or two such entries on one diagonal. 
That is, say, $M_{11}$ depends on $v$, but $M_{12}$ and $M_{21}$ does not. 
Since $M$ is degenerate it follows that then $M_{22}=0$. 
Literally as in the base of induction in Case 1 we obtain that $M$ splits.

Introduce some notation: $M_{ij}(u,v)=:M^{(1)}_{ij}(u)v+M^{(0)}_{ij}(u)$.
The induction step is performed using another induction over the minimal degree of $M^{(1)}_{ij}(u)$ among all $i,j$ such that $M^{(1)}_{ij}(u)\ne 0$. 
Assume without loss of generality that $M_{22}^{(1)}$ is nonzero and has the minimal degree, and $M^{(1)}_{21}$ is also nonzero (otherwise consider the conjugate matrix). 
Divide $M^{(1)}_{21}$ by $M^{(1)}_{22}$ from the left with remainders: $M^{(1)}_{21}=M^{(1)}_{22}X+M'_{21}$, 
where $X,M'_{21}\in\mathbb{H}[u]$ and $\deg M'_{21}<\deg M^{(1)}_{22}$. 
Subtract the second column of the matrix $M$ right-multiplied by $X$ from the first column. 
The resulting matrix is also degenerate and has either smaller number of nonzero $M^{(1)}$-s or smaller minimal degree of those 
(depending on whether $M'_{21}$ vanish or not). 
By the inductive hypothesis, the resulting matrix splits.
Thus the initial matrix $M$ splits.
\end{proof}

\noindent
\textsc{Alexey Pakharev\\
Math department, Northeastern University, Boston}
\\
\texttt{E-mail: alexey.pakharev@gmail.com}
\bigskip

\noindent
\textsc{Mikhail Skopenkov\\
National Research University Higher School of Economics,\\
and\\
Institute for information transmission problems, Russian Academy of Sciences} 
\\
\texttt{skopenkov@rambler.ru} \quad \url{http://skopenkov.ru}

\end{document}